\documentclass{article}
\usepackage{eurosym}
\usepackage{amssymb}
\usepackage{amsmath}
\usepackage{amsfonts}

\newtheorem{theorem}{Theorem}

\newtheorem{corollary}[theorem]{Corollary}

\newtheorem{definition}[theorem]{Definition}
\newtheorem{example}[theorem]{Example}

\newtheorem{lemma}[theorem]{Lemma}

\newtheorem{proposition}[theorem]{Proposition}

\newenvironment{proof}[1][Proof]{\noindent \textbf{#1.} }{\  \rule{0.5em}{0.5em}}

\begin{document}

\title{{\large Dirac operators on lightlike hypersurfaces}}
\author{{\small G\"{u}lsah Aydin Sekerci}$^{1}${\small , Abdilkadir Ceylan \c{C}\"{o}ken}$%
^{2}$ \\
%EndAName
$^{1}${\small Department of Mathematics, Faculty of Arts and Sciences,}\\
{\small S\"{u}leyman Demirel University, 32260 Isparta, Turkey}\\
$^{2}${\small Department of Mathematics, Faculty of Science,  }\\
{\small Akdeniz University, 07058 Antalya, Turkey }\\
{\footnotesize $^{1}$gulsahaydin@sdu.edu.tr, $^{2}$ceylancoken@akdeniz.edu.tr}}
\date{}
\maketitle

\begin{abstract}
In this study, we obtain a spinorial Gauss formula for a lightlike\linebreak hypersurface in Lorentzian manifold with 4-dimension. Then, we take into account the changes caused by degenerate metric on hypersurface and\linebreak investigate Dirac operator for lightlike hypersurface. Later, we\linebreak establish the relation between Dirac operators and Riemannian curvatures of manifold and hypersurface.%
\begin{equation*}
\begin{array}{c}
\end{array}%
\end{equation*}

\textbf{Key words}: Spin geometry, degenerate spin manifold, lightlike\linebreak vector.

\textbf{2010 AMS Classification}: 53C27; 57R15; 83A05
\end{abstract}

\section{Introduction}

Dirac operator has revealed due to the square root problem of \linebreak Laplacian operator in the Klein-Gordon equation. The Dirac operator, which was emerged from the studies of Paul Dirac \cite{dirac} during his  investigations on the spin-$ 1/2 $ particles like fermions and electrons, tries to find out an answer to a question that whether the first order differential equation with $ D=\sqrt{\bigtriangleup }$ exists or not. As a result of the growing attention to this equation, many\linebreak researchers from different branches such as geometricians, researchers from both\linebreak mathematical physics and analysis, have started to work on this topic.\linebreak Especially, a great amount of mathematicians also interested in this\linebreak operator after the relation between the properties of Clifford algebra and its the\linebreak coefficients of the Dirac operator. The calculation of this operator in vector spaces is relatively easy when it compared to the calculation in the manifolds. So, this operator was studied on vector spaces before considering the manifolds. Then, to  eliminate the possible problems and to ensure that the operator is well-defined in the manifold was needed some changes since vector bundles are insufficient to obtain Dirac operator on manifolds. The lack of vector  bundles was eliminated by the associated vector bundle and so, the spin geometry has been revealed. After that, Dirac operator was started to work on manifolds.  
      
While many researchers have been investigating the Dirac operators and their features on the Riemannian and Lorentzian manifolds, recently the Dirac operators on the surfaces have been attracted attention. The  investigation of the relations between the Dirac equation solutions and  immersions of surfaces in \cite{friedrich} could be given as an example. The results that exist for the Laplacian have been investigated by Hijazi and Montiel \cite{hijazi5} for the Dirac operator. In addition to the given examples, it is known that the existence of the bounds for eigenvalues of Dirac operator has a vital importance and as a result of this knowledge there exists studies \cite{Morel,hijazi2} in the literature that cover the discussions about the hypersurfaces. Moreover, Nakad and Roth \cite{nakad} aimed to develop upper bounds for the eigenvalues of Dirac operator, which is defined on the hypersurfaces of the spin manifolds. While a new upper bound for the first eigenvalue of the Dirac operator on hypersurface was examined in \cite{Ginoux}, the Dirac operator for the hypersurfaces has been discussed in all manners in many other studies \cite{hijazi1,nakad2,montiel}. In \cite{hijazi4}, some results have been shown by taking into account of the scalar curvature. The obtained results for hypersurfaces have been extended to the submanifold in \cite{hijazi3}. Also, Dereli et.al. worked on degenerate spin group and Levy-Leblond equation as given in \cite{dereli,dereli1}.

The main aim of this work is to investigate the hypersurfaces of Lorentzian spin manifolds with $ 4$-dimension. The existing studies in the literature about this problem have been mainly focused on the timelike and spacelike \linebreak hypersurfaces of Lorentzian manifolds where the lightlike hypersurfaces have left as an open problem. Those hypersurfaces has an important place in the researches due to their contributions to the applicability of the theory of  relativity. Even though lightlike geometry studies may provide many  beneficial outcomes, there exist some difficulties on working with them since they are different from many geometries. Considering these difference, we describe the spinorial Gauss formula for the lightlike  hypersurfaces and show that it is possible to reduce a spin structure to the lightlike hypersurface $M$ from the Lorentzian spin manifold $\widetilde{M}$. Then, we build up the relationship between the spinor covariant derivatives for $M$ and $\widetilde{M}$. Also, we define the Dirac operator for lightlike hypersurface by using Dirac operator of Lorentzian spin manifold. Doing those, we aim to investigate the Dirac operator and to establish a relationship between the Dirac operators of lightlike hypersurfaces and Lorentzian spin manifolds. In addition, we study special lightlike hypersurface like minimal, totally  umbilical and etc. in this represented work. 

\section{Preliminaries}     
 
In this section, definitions that will be used later have been given.
\begin{definition}
Let $V$ be a vector space over a commutative field $k$ and $Q$ be
a quadratic form on $V$. Let $T\left( V\right) =\sum\limits_{i=0}^{\infty
}\otimes ^{i}V$ denotes the tensor algebra of $V$ where $ \otimes $ is tensor product and $I_{Q}\left(
V\right) $  be the ideal in $T\left( V\right) $, which is generated by all elements
of the form $v\otimes v+Q\left( v\right) 1$ for $v\in V$. Then, the quotient
\begin{equation}
C\ell \left( V,Q\right) \equiv T\left( V\right) /I_{Q}\left( V\right)
\end{equation}
is called Clifford algebra \cite{Lawson}.
\end{definition}  
 Let us choose index $q$ with $0<q<n$ and $p=n-q$ where $ n $ is dimension of vector space. In this situation, the set of all linear isometries $ \psi: \mathbb{R}^{p,q}  \rightarrow \mathbb{R}^{p,q} $ is the same as the set of all matrices $\Psi\in GL\left( n, \mathbb{R} \right) $ which preserves scalar product on $ \mathbb{R}^{p,q} $ where $GL\left( n, \mathbb{R} \right) $ is general linear group. Then, it generates a group and is called semi-orthogonal group. Also, it is denoted by $O\left( p,q\right) $ for $p=n-q$. Also, the set 
\begin{equation}
SO\left( p,q\right) =\left\{ \Psi \in O\left( p,q\right) :\det \Psi
=1\right\}%
\end{equation}
is called special semi-orthogonal group \cite{O'Neill}. 

\begin{definition} 
Let $ \mathbb{R}^{n} $ be a $ n$- dimensional real vector space, g be a  symmetric bilinear form  on $ \mathbb{R}^{n} $ and $e_{1},e_{2},...,e_{n}$ be standard basis vectors on $ \mathbb{R}^{n} $. If the symmetric bilinear form $ g $ satisfies the condition 
\begin{equation}
g\left( e_{i},e_{j}\right) =\varepsilon _{i}\delta _{ij}, \varepsilon
_{i}=\begin{cases}
-1, & 1\leq i\leq q \\ 
1, & q\leq i\leq n%
\end{cases}, \delta _{ij}=\begin{cases}
0, & i\neq j \\ 
1, & i=j%
\end{cases}
\end{equation}
then $ g $ is called semi-Euclidean bilinear form on $ \mathbb{R}^{n} $ \cite{Baum}.
\end{definition}

We assume that $Q$ is a quadratic form for a semi-Euclidean bilinear form $ g $ on $\left(\mathbb{R}^{n},g\right) $. Then, Clifford algebra $C\ell _{p,q}:=C\ell \left(\mathbb{R}^{n},Q\right) $ is called semi-orthogonal Clifford algebra. For this Clifford algebra, there are 
\begin{align}
e_{i}^{2}=-\varepsilon _{i}, \ i=1,...,n \\ 
e_{i}e_{j}+e_{j}e_{i}=0, \ i\neq j,i,j=1,...,n
\end{align}
and $\left( 1,e_{i_{1}}\cdot ...\cdot e_{i_{s}},1\leq i_{1}<...<i_{s}\leq n,1\leq s\leq n\right) $ is the basis of $C\ell _{p,q}$ \cite{Baum}. 

A semi-orthogonal pin group is a subgroup such that
\begin{equation}
Pin\left( p,q\right) :=\left\{ a_{1}\cdot ...\cdot a_{l}:a_{i}\in
S_{q}^{n-1}\cup H_{q}^{n-1}\right\}%
\end{equation}
consists of the inverse elements of Clifford algebra $C\ell _{p,q}$ where \linebreak$ S_{q}^{n-1}=\left\{ v \in \mathbb{R}^{n}: g(v,v)=1 \right\} $ and $ H_{q}^{n-1}=\left\{ v \in \mathbb{R}^{n}: g(v,v)=-1 \right\} $. If $ l $ is even, semi-orthogonal pin group is semi-orthogonal spin group and is  denoted by $ Spin(p,q) $ \cite{Baum}.

\begin{definition}
Let $V$ be a real $4-$dimensional vector space with a symmetric bilinear
form $g$. Then, a subspace $Rad \ V$ of $V$ expressed by 
\begin{equation}
Rad \ V=\left\{ \eta \in V:g\left( \eta ,v\right) =0,v\in V\right\}
\end{equation}
is called radical space \cite{Duggal}.
\end{definition}

Also, we assume that $ V $ is a vector space with quadratic form $ Q $ and $ (V,Q) $ has rank $ n $. If $ (V,Q) $ has a radical subspace, we have $ V=V_{1}\oplus Rad \ V $ where $ dim \ V_{1}=n_{1} $, $ Rad \ V $ is the radical of $ (V,Q) $ and $ Q $ induces a quadratic form $ Q_{1} $ of $ rank \ n_{1} $ on $ V_{1} $. So, Clifford algebra, which is formed by vector space $ V $ with radical space, is called degenerate Clifford algebra and this degenerate Clifford algebra is isomorphic to the graded tensor product of $ C\ell \left( V_{1},Q_{1}\right)  $ and $ \wedge Rad \ V $ where $ \wedge $ is an exterior product. If we take $\mathbb{R}^{r,p,q}$ instead of $ V $, degenerate Clifford algebra is written as $ C\ell _{r,p,q}$ where $ r $ is the dimension of radical space in $\mathbb{R}^{r,p,q}$  \cite{crue}.   
            
\begin{proposition}
There is a homomorphism $ \sigma $ that it is defined onto the group $ \overline{T} $ of isometries of $ (V,Q) $ from Clifford group $\Gamma $ where the restriction of $ V $ to $ rad \ V $ is the identity \cite{crue}.
\end{proposition}

Degenerate pin group of $ (V,Q) $ is denoted by $ Pin(Q) $ and degenerate spin group of $ (V,Q) $ is denoted by $ Spin(Q) $. Every element of $ Pin(Q) $ is a product 
\begin{equation*}
 a_{1}\cdot ...\cdot a_{k}\cdot exp (\displaystyle\sum_{i,k}c^{ik}e_{k}f_{i}) 
\end{equation*}
where $ V=V_{1}\oplus rad \ V $, $ Q_{1} $ is a quadratic form on $ V_{1} $, $ e_{k} $ is an orthogonal basis vector of $ (V_{1},Q_{1}) $, $ f_{i} $ is an arbitrary basis vector of $  Rad \ V $ and $ Q(a_{i})=\pm1 $ for $ a_{i}\in V_{1} $. If $ k $ is even, it is an element of $ Spin(Q) $ \cite{crue}. If we take $\mathbb{R}^{r,p,q}$ instead of $ V $, the degenerate spin group is written as $ Spin(r,p,q) $. 
\section{Spinor Bundles on Lightlike Hypersurfaces}      
        
In this section, we will obtain the necessary relationships to define the\linebreak spinorial Gauss formula for the lightlike hypersurface of the Lorentzian spin manifold with $4$-dimension.

Let $\left( \widetilde{M},\widetilde{g}\right) $ be a Lorentzian spin manifold with $4$-dimension. Morever, let $ \widetilde{g} $ is given by $\widetilde{g}=\left( +,+,+,-\right) $ and $\widetilde{\nabla }$ denotes the Levi Civita connection on tangent bundle $T\widetilde{M}$. We consider an 3-dimensional lightlike hypersurface $\left(M,g\right) $ of the manifold $\left( \widetilde{M},\widetilde{g}\right) $. If there exists a vector field $\eta \neq 0$ on $M$ such that $%
g\left( \eta, X\right) =0$ for $X\in \Gamma \left( TM\right) $, then $g$
is degenerate. A subspace, which consist of tangent vector $ \eta_{x} $ at
each point $x\in M$, is called a radical or null space and it is denoted by $%
Rad \ T_{x}M$. Also, $Rad$  $TM$ is called a radical distribution of $M$ and if $M$ has the radical distribution, then it is called a lightlike hypersurface of $\widetilde{M}$. Here, induced metric $g$ by $\widetilde{g}$ is degenerate and $\nabla $ is linear connection on $M$, but it is not Levi Civita connection \cite{duggal1}.

We will show that the spin structure on manifold $\widetilde{M}$ could be reduced to lightlike hypersurface $M.$ For this, we need to the degenerate special orthogonal group to establish a relationship with the degenerate spin group. 

The basis vector of Lie algebra $so\left( 3,1\right) $ is $E_{ij}=-\varepsilon _{j}D_{ij}+\varepsilon _{i}D_{ji}$ where $D_{ij}$ denotes matrices of type $4\times 4$ whose the components of $\left( ij\right) $ are one and the other components are zero. Also, $ \varepsilon _{i} $ is the signature of vector, that is, $ \varepsilon _{i}=g(e_{i},e_{i}) $ where $ e_{i} $ is the basis vector on $\mathbb{R}^{3,1}$ \cite{Baum}. 

Also, the basis of Lie algebra $so\left( 1,2,0\right) $ for the hypersurface $M$ is 
\begin{equation}
E_{01}=
\begin{bmatrix}
0 & 0 & 0  \\ 
0 & 0 & 0  \\ 
0 & 1 & 0  
\end{bmatrix}, E_{02}=\begin{bmatrix}
0 & 0 & 0  \\ 
0 & 0 & 0  \\ 
0 & 0 & 1  
\end{bmatrix} ,E_{12}=\begin{bmatrix}
0 & 0 & -1  \\ 
0 & 1 & 0  \\ 
0 & 0 & 0  \end{bmatrix}.
\end{equation}
When we examine Lie algebras $so\left(1,2,0\right) $ and $so\left( 3,1\right) $, we see that there \linebreak exists an immersion between the Lie algebras . So, it is obvious that $SO\left( 1,2,0\right) $ is a subgroup of $SO\left( 3,1\right) $. So, we establish the relationship between special\linebreak orthogonal groups. Similar to this relationship, a  connection between spin groups is also needed. Since spin groups consist of inverse elements of\linebreak Clifford algebra, we use the features of Clifford algebra to establish this\linebreak connection. Then, there exists the following homomorphism because Clifford algebra $ C\ell _{1,2,0} $ is immersed in Clifford algebra $ C\ell _{3,1} $. 
        
\begin{lemma}
There is an algebra homomorphism between $\ C\ell _{1,2,0}$ and $C\ell _{3,1}^{0}$.
\end{lemma}

\begin{proof}
According to the universal property of the Clifford algebra, an algebra homomorphism on the Clifford algebra is found by using a linear map, which is defined between vector space and algebra. Then, the existence of such an algebra homomorphism is easily demonstrated by regarding this property. 

Let us consider the orthonormal basis $\left\{e_{1},e_{2},e_{3},e_{4}\right\}$ of $\mathbb{R}^{3,1}$ such that\linebreak $\widetilde{Q}\left( e_{4}\right) =-1$ and $\widetilde{Q}\left( e_{i}\right)=1, i=1,2,3$ where $\widetilde{Q}$ is a quadratic form for $\mathbb{R}^{3,1}$. Also, let the basis of $\mathbb{R}^{1,2,0}$ is given by $\left\{ e_{0},e_{1},e_{2} \right\} $ such that $Q\left(e_{i}\right) =1,$   $i=1,2$ and $Q\left( e_{0} \right) =0$. We assume that a map $f$ is defined by  $f: \mathbb{R}^{1,2,0}\rightarrow C\ell _{3,1}$, $f\left( v\right)=v $. This map is linear and it is necessary to provide the condition $f(v)^{2}=-Q(v)\cdot 1$. To show that, if we write $v=v_{0}e_{0}+v_{1}e_{1}+v_{2}e_{2}$ for $v\in \mathbb{R}^{1,2,0}$, then we have
\begin{align*}
f(v)^{2} & =\lbrace \frac{1}{\sqrt{2}}v_{0}(e_{3}+e_{4})+v_{1}e_{1}+v_{2}e_{2}\rbrace \lbrace \frac{1}{\sqrt{2}}v_{0}(e_{3}+e_{4})+v_{1}e_{1}+v_{2}e_{2}\rbrace =-Q(v)
\end{align*}
In this situation, the map $f$ expands to $\widetilde{f}:C\ell_{1,2,0}\longrightarrow C\ell _{3,1}$ and so, $\widetilde{f}$ is the algebra homomorphism.  
\end{proof}

It is defined by
\begin{align}
i : C\ell _{1,2,0} & \rightarrow  C\ell _{3,1} \\ \nonumber
       e_{1}          &  \mapsto     e_{1}  \\ \nonumber
        e_{2} & \mapsto  e_{2}  \\  \nonumber
   e_{0} & \mapsto  \frac{1}{\sqrt{2}}\left( e_{3}+e_{4}\right) 
\end{align}
where $e_{1},e_{2}$ and $e_{0}$ are the spacelike and lightlike basis vectors on $\mathbb{R}^{1,2,0}$, \linebreak respectively. Also, $e_{1},e_{2},e_{3}$ and $e_{4}$ are the spacelike and timelike basis \linebreak vectors on $\mathbb{R}^{3,1}$, respectively. 

Normally, when we pass from the degenerate Clifford algebra to the\linebreak nondegenerate Clifford algebra, the degenerate vectors is written as\linebreak nondegenerate, such as $e_{0}\rightarrow \frac{1}{\sqrt{2}} \left(
e_{3}+e_{4}\right) $. But, we write $ e_{0} $ for shortness.

This map, which is defined between the Clifford algebras, could be  restricted to the spin groups. Then, we obtain the following diagram
\begin{align}
Spin\left( 1,2,0\right)              & \overset{i }{\rightarrow }  Spin\left(3,1\right)  \\ \nonumber
_{\overline{\rho }}\downarrow  &   \  \  \  \  \downarrow _{\widetilde{\rho }} \\ \nonumber
SO\left( 1,2,0\right)                 & \underset{\widehat{i}}{\rightarrow }  SO\left( 3,1\right) 
\end{align}
which is commutative for adjoint maps $\overline{\rho }$ and $\widetilde{\rho }$. Also, using the relationship between $SO\left( 1,2,0\right) $ and $SO\left( 3,1\right) $, the principal bundle for $M$ is constituted by the principal bundle $ Prin_{SO\left( 3,1\right) }\widetilde{M} $. From the principal bundle for the group $SO\left( 3,1\right) $ on manifold $\widetilde{M}$, we write a map $\pi : Prin_{SO\left( 3,1\right) }\widetilde{M} \rightarrow  \widetilde{M} $ and so, there exists a diffeomorphism $\vartheta :\pi ^{-1}\left(\widetilde{U} \right) \rightarrow \widetilde{U}\times SO(3,1)$ for an open set $\widetilde{U}\subset \widetilde{M}$. Then, when this principal bundle is restricted to hypersurface $ M $, we have $\pi :  \left. Prin_{SO\left( 3,1\right) }\widetilde{M}\right\vert _{M} \rightarrow   M$ and this map is subjective. Also, we get 
\begin{equation}
\vartheta :\left. \pi ^{-1}\left( \widetilde{U}\right) \right\vert _{M}\rightarrow
\left. \widetilde{U}\times SO( 3,1) \right\vert _{M}=\left( \widetilde{U}\cap M\right)
\times SO( 1,2,0)%
\end{equation}
and so, there is the principal bundle $ \left. Prin_{SO\left( 3,1\right) }\widetilde{M}\right\vert _{M}$ since $ \vartheta $ restricted to $ M $ is a diffeomorphism. If $Prin_{SO\left( 1,2,0\right) }M $ is the principal bundle for $M$, then the relationship between the principal bundle $Prin_{SO\left( 3,1\right) }\widetilde{M}\left\vert_{M}\right. $ and the principal bundle $Prin_{SO\left( 1,2,0\right) }M $ could be established. For this, let us define a continuous map 
\begin{equation}
\xi :Prin_{SO\left( 1,2,0\right) }M\rightarrow \left. Prin_{SO\left(
3,1\right) }\widetilde{M}\right\vert _{M} .%
\end{equation}
Using the map $ \xi $, we need to show that the principal bundle with spin group on $M$ occurs if we restrict the principal bundle with spin group on $\widetilde{M}$ to $ M $. For this, according to the definition of pullback of principal bundle in \cite{Lawson}, we write the following commutative diagram since there exist the continuous map $ \xi $ and a principal bundle. So, we have
\begin{align}
\xi ^{\ast }\left( Prin_{Spin\left( 3,1\right) }\widetilde{M}\left\vert
_{M}\right. \right) =P_{Spin\left( 1,2,0\right) }M & \overset{\overline{\xi} }{%
\rightarrow }  Prin_{Spin\left( 3,1\right) }\widetilde{M}\left\vert
_{M}\right.  \\ \nonumber
_{\pi ^{\prime }}\downarrow   & \  \   \   \  \downarrow _{\pi}  \\ \nonumber
Prin_{SO\left( 1,2,0\right) }M &  \underset{\xi }{\rightarrow }  
Prin_{SO\left( 3,1\right) }\widetilde{M}\left\vert _{M}\right. 
\end{align}
where $\pi ^{\prime} $ defines principal bundle $ Prin_{Spin\left( 1,2,0\right) }M $, $ \pi $ defines principal  bundle $Prin_{Spin\left( 3,1\right) } \widetilde{M}\left\vert _{M}\right. $ 
and $\xi ^{\ast }:Prin_{Spin\left( 3,1\right) } \widetilde{M}\left\vert _{M}\right. \rightarrow Prin_{Spin\left(1,2,0\right) }M$. So, there exists the principal bundle $Prin_{Spin\left( 3,1\right) }\widetilde{M}\left\vert _{M}\right. $ such that
\begin{align}
\xi ^{\ast }\left( Prin_{Spin\left( 3,1\right) }\widetilde{M}\left\vert
_{M}\right. \right) =\lbrace
\left( x,y\right) :\xi \left( x\right) =\pi \left( y\right) , x\in Prin_{SO\left( 1,2,0\right) }M, \\ y\in Prin_{Spin\left( 3,1\right) }\widetilde{M}\left\vert
_{M}\right.\rbrace .\nonumber
\end{align}
It is seen that the bundle formed by the restriction of $  \widetilde{M}$ to $  M$ is a  principal bundle with the spin group. Thus, we show that the  restrictions of the principal bundles with special or spin groups of $  \widetilde{M}$ to $  M$ have similar properties with $  \widetilde{M}$. Then, using these results, we could define a spin structure on the restriction of $  \widetilde{M}$ to $  M$. Accordingly,  using the spinor bundle on $\widetilde{M}$, the spinor bundle on $M$ is \linebreak described. Assume that $S\widetilde{M}$ is a spinor bundle on $\widetilde{M}$, that is, \linebreak$S\widetilde{M}=Prin_{Spin^{+}\left( 3,1\right)}\widetilde{M}\times _{\widetilde{\rho }_{3,1}}\mathbb{R}^{4}$ , where $\begin{array}{c}
\widetilde{\rho }_{3,1}:Spin^{+}\left( 3,1\right) \rightarrow
Aut\left(\mathbb{R}^{4}\right) \end{array}$, $\mathbb{R}^{4}$ is a module for $C\ell _{3,1}$, $ Aut\left(\mathbb{R}^{4}\right) $ is group of automorphisms on $ \mathbb{R}^{4} $ and $ Spin^{+}\left( 3,1\right) $ is a connected component of $ Spin\left( 3,1\right) $. Locally, let $\widetilde{U}$ be an open set of $\widetilde{M}$. Then, $\phi _{r}=\left[ \widetilde{s},\alpha _{r}\right] $ is written from $S\widetilde{M}$ where $\phi _{r}\in \Gamma \left( S\widetilde{M}\right) $ is a locally section of spinor bundle, $[ , ] $ is an equivalence class and $\alpha _{r}:\widetilde{U}\rightarrow \mathbb{R}^{4}$ and $\widetilde{s}:\widetilde{U}\rightarrow Prin_{Spin\left( 3,1\right)
}\widetilde{M}$ are smooth maps. In this situation, the spinor field $\phi _{r}$ is regarded as the element of a associated bundle since each spinor bundle is actually the associated bundle. Then, there is an equivalence relation $\sim$ for $ u\in Spin\left(3,1\right) $ such that
\begin{equation}
\left[ \widetilde{s},\alpha _{r}\right] \sim \left[ \widetilde{s}u,%
\widetilde{\rho }_{3,1}\left( u^{-1}\right) \alpha _{r}\right] 
\end{equation}
where $\widetilde{s}\in Prin_{Spin\left( 3,1\right) }\widetilde{M}$ and $\alpha _{r}\in \mathbb{R}^{4}$. So, we have
\begin{equation}
\phi _{r}\left\vert _{M}\right. =\left[ \widetilde{s}\left\vert _{\widetilde{U}\cap
M}\right. ,\alpha _{r}\left\vert _{\widetilde{U}\cap M}\right. \right] 
\end{equation}
when we restrict the spinor field $\phi _{r}$ to $M$. Since $\sim $ could not be an equivalence relation for $M$, the equivalence relation $\sim $ should be revised. So, it will be $\widetilde{s}\in Prin_{Spin\left( 3,1\right) }\widetilde{M}\left\vert _{M}\right. $ when $\widetilde{s}$ is restricted to $\widetilde{U}\cap M$. Also, $u\in Spin\left( 3,1\right) $ should be the element of $Spin\left(1,2,0\right) $. In that case, if we use the map $i $, then we restate the homomorphism $\widetilde{\rho } _{3,1}$ for the spin group on $M$. It is given by
\begin{align}
Spin\left( 1,2,0\right)  &  \overset{i }{\rightarrow }  
Spin\left( 3,1\right)  \\ \nonumber
 _{\widetilde{\rho }_{3,1}\circ i }\searrow  \  \ &  \  \  \  \
\downarrow  ^{\widetilde{\rho }_{3,1}} \\  \nonumber
  &  Aut\left(\mathbb{R}^{4}\right) 
\end{align}
So, we obtain the homomorphism for the group $Spin\left( 1,2,0\right)$. Then, we write equivalence relation, which gives the spinor bundle on $M$. Thus, we have 
\begin{equation}
\left[ \widetilde{s}\left\vert _{\widetilde{U}\cap M}\right. ,\alpha_{r}\left\vert _{\widetilde{U}\cap M}\right. \right] \sim \left[ \widetilde{s}\left\vert _{\widetilde{U}\cap M}\right. u,\left( \widetilde{\rho }_{3,1}\circ i \right) \left( u^{-1}\right) \alpha _{r}\left\vert _{\widetilde{U}\cap M}\right. \right] 
\end{equation}
for $u\in Spin\left( 1,2,0\right)$. From there, we find 
\begin{equation}
S\widetilde{M}\left\vert _{M}\right. =Prin_{Spin^{+}\left( 1,2,0\right)
}M\times _{\widetilde{\rho }_{3,1}\circ i }\triangle _{M}%
\end{equation}
where $ Spin^{+}\left( 1,2,0\right) $ is connected component of $ Spin\left( 1,2,0\right) $ and $ \triangle _{M} $ is a\linebreak module of representation $ \widetilde{\rho }_{3,1}\circ i $.

Now, let us express the Clifford multiplication for the hypersurface. So, reduced Clifford multiplication from $\widetilde{M}$ to $M$ is obtained as following since $\rho _{3,1}\circ i $ provides the Clifford multiplication for $M$. It is defined by 
\begin{align}
\rho _{3,1}:  C\ell _{3,1} & \rightarrow   Hom\left(\mathbb{R}^{4}, \mathbb{R}^{4}\right)  \\  \nonumber
 \phi  & \mapsto   \widetilde{\rho }_{3,1}\left( \phi \right) \left(
v\right) \equiv \phi \cdot v
\end{align}
for $\phi \in C\ell _{3,1}$ and $v\in \mathbb{R}^{4}$. So, we obtain that
\begin{align}
C\ell _{1,2,0} & \overset{i }{\rightarrow }  C\ell _{3,1}   \overset
{\widetilde{\rho }_{3,1}}{\rightarrow }  Hom\left(\mathbb{R}^{4},%
\mathbb{R}^{4}\right)  \\  \nonumber
\phi  & \mapsto  \   \phi   \  \mapsto   \phi \cdot v
\end{align}

\section{Spinorial Gauss formula for lightlike \\ hypersurfaces} 

Let $\widetilde{M}$ be a $ 4$-dimensional Lorentzian spin manifold and $M$ be a lightlike\linebreak hypersurface in $\widetilde{M}.$ Then, a complementary vector bundle $S\left( TM\right) $ of $Rad \ TM$ in $TM$ is called a screen distribution on $M$
and there exists\linebreak $TM=Rad \ TM\perp S\left( TM\right) .$ Morever, we have the following decompositions.
\begin{equation}
T\widetilde{M}=S\left( TM\right) \perp \left( Rad \ TM\oplus ltr\left(
TM\right) \right) =TM\oplus ltr\left( TM\right)
\end{equation}
where $ltr\left( TM\right) $ is a complementary vector bundle to $TM$ in $T\widetilde{M}$ and it is called lightlike transversal bundle of $M$. 

Let the locally orthonomal frame of the tangent bundle $T\widetilde{M}$ be  $\left\{ s_{1},s_{2},s_{3},s_{4}\right\} $ such that $\left\{ s_{1},s_{2},s_{3}\right\} $ and $\left\{ s_{4}\right\} $ are spacelike and timelike  vectors according to $ \widetilde{g} $, respectively. Considering these vectors, it is  possible to construct lightlike vectors. We write that $
s_{0}=\frac{1}{\sqrt{2}}\left( s_{3}+s_{4}\right)
$ and $
N=\frac{1}{\sqrt{2}}\left( s_{3}-s_{4}\right)
$ where these vectors satisfy the conditions 
\begin{align}
g\left( s_{0},N\right) =1,  g\left( s_{0},s_{i}\right) =g\left( N,s_{i}\right) =0,i=1,2.
\end{align}
Thus, the quasi orthonormal basis of $\widetilde{M}$ is given by $\left\{ N,s_{0},s_{1},s_{2}\right\} $ and the basis of $3$-dimensional lightlike subbundle $TM$ of $T\widetilde{M}$ is $\left\{s_{0},s_{1},s_{2}\right\} $ and $N$ is a normal vector field for the hypersurface $M$. 

We write Gauss-Weingarten formula for lightlike
case to obtain the induced geometric objects. Let $ \widetilde{\nabla } $ be Levi Civita connection on $ \widetilde{M} $ and $ \nabla $ be a linear connection on $ M $. So, we have 
\begin{align}
\widetilde{\nabla }_{X}Y &=\nabla _{X}Y+h\left( X,Y\right) N, \\
\widetilde{\nabla }_{X}N &=-A_{N}\left( X\right) +\nabla _{X}^{t}N  \notag
\end{align}
for $X,Y\in \Gamma \left( TM\right)$ where $\nabla _{X}Y, A_{N}\left( X\right) \in \Gamma \left( TM\right) $ and $N, \nabla _{X}^{t}N\in ltr\left( TM\right)$. Morever, $h$ is a symmetric bilinear form on $\Gamma \left( TM\right)$, $A_{N}$ is a shape operator of $M$ in $\widetilde{M}$ and $\nabla ^{t}$ is a linear connection on $ltr\left(TM\right)$ \cite{duggal1}.

Also, if $ f^{\prime } $ is defined by
\begin{align}
\label{eq1}
f^{\prime }:  TM & \rightarrow  T^{\ast }M \\  \nonumber
 X & \mapsto  f\left( X\right)(Y) =g\left( PX,Y\right) +\eta \left(
X\right) \eta \left( Y\right)
\end{align}
then it is an isomorphism where $P$ is projection morphism of $\Gamma\left( TM\right) $ on \linebreak$\Gamma \left( S\left( TM\right)\right) $ and $\eta$ is $1$-form defined by $\eta \left( X\right) =\widetilde{g}\left( N,X\right) $ \cite{atin}.

Now, we get the spinorial Gauss formula for the lightlike hypersurface with these informations. Let $S\widetilde{M}, SM $ be spinor bundles of $\widetilde{M}, M,$ and spinorial connections on the spinor bundles $S\widetilde{M}, SM $ are denoted by $\widetilde{\nabla }^{s}, \nabla ^{s}$, respectively. The connection on the spinor bundle for causal structure $\left(3,1\right) $ is given by 
\begin{equation}
\label{eq2}
\widetilde{\nabla }_{X}^{s}\Phi =X\left( \Phi \right) +\dfrac{1}{2}%
\displaystyle\sum_{i<j=1}^{4}\varepsilon _{i}\varepsilon _{j}\widetilde{g}\left( 
\widetilde{\nabla }_{X}s_{i},s_{j}\right) s_{i}\cdot s_{j}\cdot \Phi%
\end{equation}
for $X\in \Gamma \left( TM\right) ,\Phi \in \Gamma \left( S\widetilde{M}\right) $ \cite{Baum}.

\begin{theorem}
Let $\widetilde{M}$ be $4$-dimensional Lorentzian spin manifold with a metric tensor $\widetilde{g}=\left( +,+,+,-\right) $ and connection on the spinor bundle $S\widetilde{M}$ be $\widetilde{\nabla }^{s}$. We assume that $\left( M,g\right) $ is a $3$-dimensional lightlike hypersurface of $\left( \widetilde{M},\widetilde{g}\right) $ and $\nabla ^{s}$ is the connection on spinor bundle $SM$. The relation between these connections is given by 
\begin{equation}
\widetilde{\nabla }_{X}^{s}\varphi =\nabla _{X}^{s}\varphi +\dfrac{1}{2}%
\displaystyle\sum_{i=2}^{2}h\left( X,s_{i}\right) s_{i}\cdot N\cdot \varphi%
\end{equation}
for $X\in \Gamma \left( TM\right) $, $\varphi \in \Gamma \left(SM\right) $. Here, $N$ is a normal vector field on $M$, $s_{i}$ is a locally orthonormal basis vector field on $TM $ and $h$ is a symmetric bilinear form. 
\end{theorem}

\begin{proof}
 If we write more clearly (\ref{eq2}), we have 
\begin{align*}
\widetilde{\nabla }_{X}^{s}\Phi =X\left( \Phi \right) +\dfrac{1}{2} \biggl(&\widetilde{g}\left( \widetilde{\nabla }_{X}s_{1},s_{2}\right) s_{1}\cdot s_{2}\cdot \Phi +\widetilde{g}\left( \widetilde{\nabla }_{X}s_{1},s_{3}\right) s_{1}\cdot s_{3}\cdot \Phi \\ 
&-\widetilde{g}\left( \widetilde{\nabla }_{X}s_{1},s_{4}\right) s_{1}\cdot s_{4}\cdot \Phi +\widetilde{g}\left( \widetilde{\nabla }_{X}s_{2},s_{3}\right) s_{2}\cdot s_{3}\cdot \Phi \\ 
&-\widetilde{g}\left( \widetilde{\nabla }_{X}s_{2},s_{4}\right) s_{2}\cdot s_{4}\cdot \Phi -\widetilde{g}\left( \widetilde{\nabla }_{X}s_{3},s_{4}\right) s_{3}\cdot s_{4}\cdot \Phi )\biggr)
\end{align*}
for $X\in \Gamma \left( TM\right) $, $\varphi \in \Gamma \left(SM\right) $ where $s_{i}$ is a locally orthonormal basis vector field on $T\widetilde{M} $ and $\widetilde{\nabla }^{s}$ is the connection on the spinor bundle $S\widetilde{M}$. If we use Gauss formula and write $s_{3}=\dfrac{1}{\sqrt{2}}\left( s_{0}+N\right) $, $s_{4}=\dfrac{1}{\sqrt{2}}\left( s_{0}-N\right) $ instead of $s_{3}, s_{4}$, then we have 
\begin{align*}
\left( \widetilde{\nabla }_{X}^{s}\Phi \right) \left\vert _{M}\right.  
=& X\left( \Phi \right) \left\vert _{M}\right. +\dfrac{1}{2}g\left( \nabla
_{X}s_{1},s_{2}\right) s_{1}\cdot s_{2}\cdot \Phi \left\vert _{M}\right.  \\ 
&+\dfrac{1}{4}\biggl[ 
\left\{ \left[ g\left( \nabla _{X}s_{1},N\right) +h\left( X,s_{1}\right) %
\right] s_{1}\cdot \left( s_{0}+N\right) \cdot \Phi \left\vert _{M}\right.
\right\} \\ 
& -\left\{ \left[ h\left( X,s_{1}\right) -g\left( \nabla _{X}s_{1},N\right) %
\right] s_{1}\cdot \left( s_{0}-N\right) \cdot \Phi \left\vert _{M}\right.
\right\} \\ 
&+\left\{ \left[ g\left( \nabla _{X}s_{2},N\right) +h\left( X,s_{2}\right) %
\right] s_{2}\cdot \left( s_{0}+N\right) \cdot \Phi \left\vert _{M}\right.
\right\} \\ 
&-\left\{ \left[ -g\left( \nabla _{X}s_{2},N\right) +h\left( X,s_{2}\right) %
\right] s_{2}\cdot \left( s_{0}-N\right) \cdot \Phi \left\vert _{M}\right.
\right\} \biggr].
\end{align*}
Also, we find that 
\begin{align*}
\widetilde{\nabla }_{X}^{s}\left( \Phi \left\vert _{M}\right. \right)
=X\left( \Phi \left\vert _{M}\right. \right) +\dfrac{1}{2}& \biggl[ 
g\left( \nabla _{X}s_{1},s_{2}\right) s_{1}\cdot s_{2}\cdot \Phi \left\vert
_{M}\right.+g\left( \nabla _{X}s_{1},N\right) s_{1}\cdot s_{0}\cdot \Phi \left\vert
_{M}\right. \\ 
&+h\left( X,s_{1}\right) s_{1}\cdot N\cdot \Phi \left\vert _{M}\right.  +h\left( X,s_{2}\right) s_{2}\cdot N\cdot \Phi \left\vert _{M}\right.\\ 
&+g\left( \nabla _{X}s_{2},N\right) s_{2}\cdot s_{0}\cdot \Phi \left\vert
_{M}\right.
\biggr]
\end{align*}
from $X\left( \Phi \right) \left\vert _{M}\right. =X\left( \Phi \left\vert _{M}\right. \right) $ and $\left( \widetilde{\nabla } _{X}^{s}\Phi \right) \left\vert _{M}\right. =\widetilde{\nabla } _{X}^{s}\left( \Phi \left\vert _{M}\right. \right) $. If we show as\linebreak $\Phi \left\vert _{M}\right. =\varphi $, we obtain
\begin{align*}
\widetilde{\nabla }_{X}^{s}\varphi =X\left( \varphi \right)  +\dfrac{1}{2}%
\biggl[& g\left( \nabla _{X}s_{1},s_{2}\right) s_{1}\cdot s_{2}\cdot \varphi +g\left(
\nabla _{X}s_{1},N\right) s_{1}\cdot s_{0}\cdot \varphi \\ &+h\left( X,s_{1}\right) s_{1}\cdot N\cdot \varphi +h\left( X,s_{2}\right)
s_{2}\cdot N\cdot \varphi \\ 
&+g\left( \nabla _{X}s_{2},N\right) s_{2}\cdot s_{0}\cdot \varphi
\biggr].
\end{align*}
So, we get
\begin{equation*}
\widetilde{\nabla }_{X}^{s}\varphi =\nabla _{X}^{s}\varphi +\dfrac{1}{2}%
\displaystyle\sum_{i=1}^{2}h\left( X,s_{i}\right) s_{i}\cdot N\cdot \varphi%
\end{equation*}
since the covariant derivative on spinor bundle $SM$  is 
\begin{align*}
\nabla _{X}^{s}\varphi =X\left( \varphi \right) +\dfrac{1}{2}\biggl[&
-g\left( \nabla _{X}s_{1},N\right) s_{0}\cdot s_{1}\cdot \varphi -g\left(
\nabla _{X}s_{2},N\right) s_{0}\cdot s_{2}\cdot \varphi \\ &+g\left( \nabla _{X}s_{1},s_{2}\right) s_{1}\cdot s_{2}\cdot \varphi
\biggr].
\end{align*}
\end{proof}            
So, the obtained this formula is called spinorial Gauss formula for lightlike hypersurfaces.
\begin{theorem}
Let $\widetilde{M}$ be a $4$-dimensional Lorentzian spin manifold whose\linebreak Riemannian curvature is denoted by $\widetilde{R}$ and $M$ be a hypersurface of $\widetilde{M}$ whose Riemannian curvature associated with spinor bundle is
denoted by $R$. The\linebreak relationship between their Riemannian curvatures is as
the following.
\begin{align}
\widetilde{R}\left( X,Y\right) \varphi =& R\left( X,Y\right) \varphi
-g\left( R\left( X,Y\right) s_{0},N\right) s_{0}\cdot N\cdot \varphi \\  \nonumber
& +\left[ g\left( \nabla _{X}s_{0},A_{N}\left( Y\right) \right)
-g\left( \nabla _{Y}s_{0},A_{N}\left( X\right) \right) \right] s_{0}\cdot
N\cdot \varphi \\  \nonumber
& +\left[ g\left( \nabla _{X}s_{0},N\right) \nabla
_{Y}s_{0}-g\left( \nabla _{Y}s_{0},N\right) \nabla _{X}s_{0}\right] N\cdot
\varphi \\  \nonumber
&+\left[ g\left( \nabla _{X}s_{0},N\right) \widetilde{\nabla }%
_{Y}N-g\left( \nabla _{Y}s_{0},N\right) \widetilde{\nabla }_{X}N\right]
s_{0}\cdot \varphi .
\end{align}
\end{theorem}

\begin{proof}
We assume that $s_{i}$ is locally frame field for $U\subset M$ , $N$ is a normal vector field on $M$ and $h$ is symmetric bilinear form, which is coefficient of the second fundamental form. So, for $X,Y\in \Gamma \left(TM\right) $ and $\varphi \in \Gamma \left( SM\right) $, we have
\begin{align*}
\widetilde{R}\left( X,Y\right) \varphi  = & \nabla _{X}^{s}\left( \nabla _{Y}^{s}\varphi \right) +\dfrac{1}{2}%
\displaystyle\sum_{i=1}^{2}h\left( X,s_{i}\right) s_{i}\cdot N\cdot \nabla
_{Y}^{s}\varphi \\ 
&+\dfrac{1}{2}\displaystyle\sum_{i=1}^{2} \biggl[ 
\left( \widetilde{\nabla }_{X}h\right) \left( Y,s_{i}\right) s_{i}\cdot
N\cdot \varphi+h\left( \widetilde{\nabla }_{X}Y,s_{i}\right) s_{i}\cdot
N\cdot \varphi  \\ 
&+h\left( Y,\widetilde{\nabla }_{X}s_{i}\right) s_{i}\cdot N\cdot \varphi+h\left( Y,s_{i}\right) \widetilde{\nabla }_{X}s_{i}\cdot N\cdot \varphi  \\ 
&+h\left( Y,s_{i}\right) s_{i}\cdot \widetilde{\nabla }_{X}N\cdot \varphi+h\left( Y,s_{i}\right) s_{i}\cdot N\cdot \widetilde{\nabla }_{X}^{s}\varphi 
\biggr]  \\ 
& -\nabla _{Y}^{s}\left( \nabla _{X}^{s}\varphi \right) -\dfrac{1%
}{2}\displaystyle\sum_{i=1}^{2}h\left( Y,s_{i}\right) s_{i}\cdot N\cdot \nabla
_{X}^{s}\varphi \\ 
&-\dfrac{1}{2}\displaystyle\sum_{i=1}^{2} \biggl[
\left( \widetilde{\nabla }_{Y}h\right) \left( X,s_{i}\right) s_{i}\cdot
N\cdot \varphi+h\left( \widetilde{\nabla }_{Y}X,s_{i}\right) s_{i}\cdot
N\cdot \varphi  \\ 
&+h\left( X,\widetilde{\nabla }_{Y}s_{i}\right) s_{i}\cdot N\cdot \varphi+h\left( X,s_{i}\right) \widetilde{\nabla }_{Y}s_{i}\cdot N\cdot \varphi  \\ 
&+h\left( X,s_{i}\right) s_{i}\cdot \widetilde{\nabla }_{Y}N\cdot \varphi+h\left( X,s_{i}\right) s_{i}\cdot N\cdot \widetilde{\nabla }_{Y}^{s}\varphi 
\biggr]  -\widetilde{\nabla }_{\widetilde{\nabla }_{X}Y-\widetilde{%
\nabla }_{Y}X}^{s}\varphi .
\end{align*}
If we use Gauss-Weingarten equations, then $\widetilde{\nabla }_{X}Y-\widetilde{\nabla }_{Y}X=\nabla _{X}Y-\nabla _{Y}X$ and $N\cdot N=0$. Thus, we obtain
\begin{align*}
\widetilde{R}\left( X,Y\right) \varphi = & R\left( X,Y\right) \varphi +
\dfrac{1}{2}\displaystyle\sum_{i=1}^{2}\biggl[ 
X\left( h\left( Y,s_{i}\right) \right)-Y\left( h\left( X,s_{i}\right) \right)\biggr] s_{i}\cdot N\cdot \varphi   \\ 
&+\dfrac{1}{2}\displaystyle\sum_{i=1}^{2}\biggl[ 
-h\left( \left[ X,Y\right] ,s_{i}\right) s_{i}\cdot N\cdot \varphi +h\left(
Y,s_{i}\right) \nabla _{X}s_{i}\cdot N\cdot \varphi  \\ 
& -h\left( X,s_{i}\right) \nabla _{Y}s_{i}\cdot N\cdot \varphi +h\left(
Y,s_{i}\right) s_{i}\cdot \widetilde{\nabla }_{X}N\cdot \varphi \\ 
& -h\left( X,s_{i}\right) s_{i}\cdot \widetilde{\nabla }_{Y}N\cdot \varphi 
\biggr] .
\end{align*}
\end{proof} 

\begin{theorem}
Let $\widetilde{M}$ be $4$-dimensional Lorentzian spin manifold and $M$ be a lightlike hypersurface. If $ M $ is a totally geodesic, the spinor covariant derivative of hypersurface $M$ and manifold $\widetilde{M}$ are the same. 
\end{theorem}

\begin{theorem}
Let $\widetilde{M}$ be $4$-dimensional Lorentzian spin manifold and $M$ be a lightlike hypersurface. If $ M $ is a totally umbilical, there exists the relation
\begin{align}
s_{k}=s_{0}\text{ ise, }\widetilde{\nabla }_{s_{k}}^{s}\varphi &=\nabla
_{s_{k}}^{s}\varphi \\ \nonumber
s_{k}\neq s_{0}\text{ ise, }\widetilde{\nabla }_{s_{k}}^{s}\varphi &=\nabla
_{s_{k}}^{s}\varphi +\dfrac{1}{2}\varepsilon_{k}c_{k}s_{k}\cdot N\cdot \varphi
\end{align}
between the spinor covariant derivative of $M$ and $\widetilde{M}$. Here, $c_{k}$ is constant, $s_{i}$ is a locally frame field for open set $U\subset M$ and $N$ is the normal vector field on $M$.
\end{theorem}

\section{Dirac Operator for Lightlike Hypersurfaces} 

\begin{theorem}
Let $\left( \widetilde{M},\widetilde{g}\right) $ be $4$-dimensional Lorentzian spin manifold and $M$ be a lightlike hypersurface of $\widetilde{M}$. Dirac operator reduced by $\widetilde{M}$ on $M$ is given as
\begin{equation}
D=\displaystyle\sum_{i=1}^{2} s_{i}\cdot \nabla _{s_{i}}^{s}+ s_{0}\cdot \nabla _{s_{0}}^{s}%
\end{equation}
where $s_{i}$ is locally frame field for $U\subset M$, $s_{0}$ is lightlike vector field on $TM\left\vert _{U}\right. $, $N$ is a normal vector field on $M$  and $\nabla ^{s}$ is connection on spinor bundle $SM$.
\end{theorem}

\begin{proof}
Dirac operator is defined by 
\begin{equation*}
D:\Gamma \left( SM\right) \overset{\nabla ^{s}}{\rightarrow }\Gamma \left(
T^{\ast }M\otimes SM\right) \overset{f}{\rightarrow }\Gamma \left( TM\otimes
SM\right) \overset{\mu }{\rightarrow }\Gamma \left( SM\right)%
\end{equation*}
where $\nabla ^{s}$ is connection on spinor bundle $SM$, $\mu $ is Clifford multiplication and $ f $ is a map $f:\Gamma \left( T^{\ast }M\otimes S\right) \overset{}{\rightarrow }\Gamma\left( TM\otimes S\right) $. It should be an isomorphism to pass between these maps. In this situation, if $ f^{\prime } $ is defined by
\begin{align*}
f^{\prime }:  TM & \rightarrow T^{\ast }M \\ 
 X & \rightarrow  f^{\prime }\left( X\right)(Y) =g\left( PX,Y\right) +\eta \left(
X\right) \eta \left( Y\right)
\end{align*}
then it is isomorphism. According to (\ref{eq1}), $P$ is projection morphism of $\Gamma\left( TM\right) $ on $\Gamma \left( S\left( TM\right)\right) $ and $\eta$ is $1$-form defined by $\eta \left( X\right) =\widetilde{g}\left( N,X\right) $. Let $\left\{ s_{0},s_{1},s_{2}\right\} $ be a locally basis field on $U\subset M$ . Then, $f^{\prime }$ is given by 
\begin{align*}
f^{\prime }:  TM & \rightarrow  T^{\ast }M \\ 
 s_{i} & \mapsto  w^{i}
\end{align*}
for the basis vector fields where $\left( w^{i}\right) $ is dual basis of $\left(s_{i}\right) $ for $i=0,1,2$. So, we write $f: T^{\ast }M \rightarrow TM $ since $f^{\prime }$ is isomorphism. Also, the condition $w^{i}\left(s_{j}\right) =\delta _{ij}$ should be satisfied.
 \begin{itemize}
  \item 
 For $ i=1,2$, we find 
\begin{align*}
w^{i}\left( s_{i}\right) & =g\left( P\left( f\left( w^{i}\right) \right)
,s_{i}\right) +\eta \left( f\left( w^{i}\right) \right) \eta \left(
s_{i}\right) \\
&=g\left( f\left( w^{i}\right) ,s_{i}\right) +\widetilde{g}\left( N,f\left(
w^{i}\right) \right) \widetilde{g}\left( N,s_{i}\right) \\
&=g\left( f\left( w^{i}\right) ,s_{i}\right)
\end{align*}
So, we have $ f\left( w^{i}\right)=\varepsilon _{i}s_{i} $ from $g\left( f\left( w^{i}\right) ,s_{i}\right) =1$. 
 \item For $ i=0$, we obtain 
\begin{align*}
w^{0}\left( s_{0}\right) & =g\left( P\left( f\left( w^{0}\right) \right)
,s_{0}\right) +\eta \left( f\left( w^{0}\right) \right) \eta \left(
s_{0}\right) \\
& =g\left( f\left( w^{0}\right) ,s_{0}\right) +\widetilde{g}\left( N,f\left(
w^{0}\right) \right) \widetilde{g}\left( N,s_{0}\right)\\
&=\widetilde{g}\left( N,f\left( w^{0}\right) \right)
\end{align*}
So, we have $f\left( w^{0}\right) =s_{0}$.
\end{itemize}
 Then, Dirac operator is given by 
\begin{align*}
D & =\mu \circ f\circ \nabla _{s_{i}}^{s}\\
&=\mu \circ f\left( w^{i}\otimes
\nabla _{s_{i}}^{s}\right) \\
&=\mu \left( f\left( w^{i}\right) \otimes \nabla
_{s_{i}}^{s}\right) \\
&=\displaystyle\sum_{i=1}^{2}\varepsilon _{i} s_{i}\cdot \nabla
_{s_{i}}^{s}+ s_{0}\cdot \nabla _{s_{0}}^{s} \\
& =\displaystyle\sum_{i=1}^{2} s_{i}\cdot \nabla
_{s_{i}}^{s}+ s_{0}\cdot \nabla _{s_{0}}^{s} .
\end{align*}
\end{proof}       
       
\begin{theorem}
Let $\widetilde{M}$ be a $4$-dimensional Lorentzian spin manifold whose Dirac operator is denoted by $\widetilde{D}$ and $M$ be a hypersurface of $\widetilde{M}$ whose Dirac operator is denoted by $D$. The
relationship between their Dirac operators is 
\begin{equation}
D\varphi =\widetilde{D}\varphi -s_{0}\cdot \widetilde{\nabla }%
_{N}^{s}\varphi +(s_{0}-N)\cdot \widetilde{\nabla }%
_{s_{0}}^{s}\varphi -\displaystyle\sum_{k=1}^{2}h\left( s_{k},s_{0}\right)
s_{0}\cdot s_{k} \cdot N\cdot \varphi +HN\cdot \varphi%
\end{equation}
for any $\varphi\in \Gamma(SM) $ where $s_{i}$ is locally frame field for $U\subset M$, $ H $ is mean curvature and $\widetilde{\nabla }^{s}$ is connection on spinor bundle $S\widetilde{M}$.  
\end{theorem}
\begin{proof} 
From Dirac operator and spinorial Gauss formula, we obtain 
\begin{align*}
D\varphi =&\displaystyle\sum_{i=1}^{2} s_{i} \cdot \nabla
_{s_{i}}^{s}\varphi + s_{0} \cdot \nabla
_{s_{0}}^{s}\varphi \\ 
=& \displaystyle\sum_{i=1}^{2} s_{i} \cdot \left[ \widetilde{%
\nabla }_{s_{i}}^{s}\varphi -\dfrac{1}{2}\displaystyle\sum_{k=1}^{2}\left[ h\left(
s_{k},s_{i}\right) s_{k}\cdot N\cdot \varphi \right] \right]\\
& + s_{0} \cdot \left[ \widetilde{%
\nabla }_{s_{0}}^{s}\varphi -\dfrac{1}{2}\displaystyle\sum_{k=1}^{2}\left[ h\left(
s_{k},s_{0}\right) s_{k}\cdot N\cdot \varphi \right] \right] \\ 
=& \displaystyle\sum_{i=1}^{2}s_{i} \cdot \widetilde{\nabla }_{s_{i}}^{s}\varphi-\dfrac{1}{2}\displaystyle\sum_{i,k=1}^{2}h\left( s_{k},s_{i}\right)
s_{i}\cdot s_{k} \cdot N\cdot \varphi +s_{0}\cdot \widetilde{\nabla }%
_{s_{0}}^{s}\varphi\\
&-\dfrac{1}{2}\displaystyle\sum_{k=1}^{2}h\left( s_{k},s_{0}\right)
s_{0}\cdot s_{k}\cdot N\cdot \varphi 
\end{align*}
for $\varphi \in \Gamma \left( SM\right) $. If we add and substract $s_{3}\cdot \widetilde{\nabla }_{s_{3}}^{s}\varphi -s_{4}\cdot \widetilde{\nabla } _{s_{4}}^{s}\varphi $ to this equation, then we find
\begin{align*}
D\varphi & =\displaystyle\sum_{i=1}^{2}s_{i}\cdot \widetilde{\nabla }%
_{s_{i}}^{s}\varphi +s_{0}\cdot \widetilde{\nabla }_{N}^{s}\varphi +N\cdot \widetilde{\nabla }_{s_{0}}^{s}\varphi-s_{0}\cdot \widetilde{\nabla }_{N}^{s}\varphi -N\cdot \widetilde{\nabla }_{s_{0}}^{s}\varphi  \\ 
&  \ -\dfrac{1}{2}\displaystyle\sum_{i,k=1}^{2}h\left( s_{k},s_{i}\right)
s_{i}\cdot s_{k} \cdot N\cdot \varphi+s_{0}\cdot \widetilde{\nabla }%
_{s_{0}}^{s}\varphi -\dfrac{1}{2}\displaystyle\sum_{k=1}^{2}h\left( s_{k},s_{0}\right)
s_{0}\cdot s_{k}\cdot N\cdot \varphi .
\end{align*}
Using $H=\dfrac{1}{n}\sum\limits h\left( e_{i},e_{i}\right) $ and $s_{i}s_{k}=-s_{k}s_{i}$, we have
\begin{equation*}
D\varphi =\widetilde{D}\varphi -s_{0}\cdot \widetilde{\nabla }%
_{N}^{s}\varphi +(s_{0}-N)\cdot \widetilde{\nabla }%
_{s_{0}}^{s}\varphi -\displaystyle\sum_{k=1}^{2}h\left( s_{k},s_{0}\right)
s_{0}\cdot s_{k} \cdot N\cdot \varphi +HN\cdot \varphi .%
\end{equation*}
\end{proof}

\begin{corollary}
Let $\widetilde{M}$ be $4$-dimensional Lorentzian spin manifold and $M$ be lightlike hypersurface of $\widetilde{M}$. If $M$ is a minimal hypersurface, the relation between Dirac operators of $M$ and $\widetilde{M}$ is 
\begin{equation}
D\varphi =\widetilde{D}\varphi -s_{0}\cdot \widetilde{\nabla }%
_{N}^{s}\varphi +(s_{0}-N)\cdot \widetilde{\nabla }%
_{s_{0}}^{s}\varphi -\displaystyle\sum_{k=1}^{2}h\left( s_{k},s_{0}\right)
s_{0}\cdot s_{k} \cdot N\cdot \varphi %
\end{equation}
for $\varphi \in \Gamma \left( SM\right) $.
\end{corollary}

\begin{example}
Let $\left(\mathbb{R}^{3,1},\widetilde{g}\right) $ be the Minkowski space with signature  $\left( +,+,+,-\right) $ of the canonical basis $\left( \partial_{1},\partial _{2},\partial _{3},\partial _{4}\right) $. $\left( M,g\right) $
is the lightlike hypersurface given by 
\begin{equation*}
M=\left\{ \left( -x,y-z,-y-z,-x\right) \in \mathbb{R}_{1}^{4}:x,y,z\in \mathbb{R} \right\} 
\end{equation*}
Then, $Rad \ TM$ and $ltr\left( TM\right) $ are defined by 
\begin{align*}
Rad \ TM &=Sp\left\{ s_{0}=-\partial _{0}-\partial _{3}\right\}  \\
ltr\left( TM\right)  &=Sp\left\{ N=-\partial _{0}+\partial _{3}\right\} 
\end{align*}
So, the screen distribution $S\left( TM\right) $ is spanned by
\begin{equation*}
s_{1}=\partial _{1}-\partial _{2},s_{2}=-\partial _{1}-\partial _{2}
\end{equation*}
In this situation, we obtain the vector fields $N,s_{0},s_{1},s_{2}$
satisfying the following conditions.
\begin{equation*}
g\left( s_{0},N\right) =1,g\left( s_{0},s_{i}\right) =g\left( N,s_{i}\right)
=0,i=1,2
\end{equation*}
Then, we obtain that for $i,j=0,1,2$, 
\begin{align*}
h\left( s_{i},s_{j}\right) =\widetilde{g}\left( \widetilde{\nabla }%
_{s_{i}}s_{j},s_{0}\right) =0
\end{align*}
Thus, relation between the spinorial covariant derivatives
\begin{equation*}
\widetilde{\nabla }_{s_{i}}^{s}\varphi =\nabla _{s_{i}}^{s}\varphi +\dfrac{1%
}{2}\displaystyle\sum_{j=1}^{2}h\left( s_{i},s_{j}\right) s_{j}\cdot N\cdot
\varphi \Longrightarrow \widetilde{\nabla }_{s_{i}}^{s}\varphi =\nabla
_{s_{i}}^{s}\varphi 
\end{equation*}
for $\varphi \in \Gamma \left( SM\right) $ and $s_{i},i=0,1,2$.
\end{example}

\textbf{\bigskip .}

\end{document}